\documentclass{amsart}
 \usepackage{amsthm,amsfonts,amsmath,amssymb,latexsym,epsfig,upref,eucal,ae}
\usepackage[all]{xy}

\newtheorem{theorem}[subsubsection]{Theorem}
\newtheorem{lemma}[subsubsection]{Lemma}
\newtheorem{proposition}[subsubsection]{Proposition}

\newenvironment{remark}{\medskip \refstepcounter{subsubsection}
\noindent  {\bf Remark \thetheorem}.\rm}{\,}
\newenvironment{definition}{\medskip \refstepcounter{subsubsection}
\noindent  {\bf Definition \thetheorem}.\rm}{\,}

\newtheorem{theointro}{Theorem}

\newtheorem{corintro}[theointro]{Corollary}

\def\mb#1{{\mathbb #1}}
\def\mc#1{{\mathcal #1}}

\def\ff{I\hspace{-0.45mm}I}
\def\over#1{\overline{#1}}

\def\db{\overline{\partial}}
\def\hok{\mbox{}\begin{picture}(10,10)\put(1,0){\line(1,0){7}}
\put(8,0){\line(0,1){7}}\end{picture}\mbox{}}
\def\BOne{{\mathchoice {\rm 1\mskip-4mu l} {\rm 1\mskip-4mu l}
                          {\rm 1\mskip-4.5mu l} {\rm 1\mskip-5mu l}}}

\def\Aut{\mathrm{Aut}}

\def\grad{\mathrm{grad}\,}

\def\PSL{\mathbb{PSL}}
\def\SS{\mathbb{S}}
\def\ZZ{\mathbb{Z}}
\def\SL{\mathbb{SL}}

\def\PP{\mathbb{P}}
\def\vol{\mathrm{vol}}

\def\Isom{\mathrm{Isom}}

\def\cM{\mathcal{M}}

\def\cN{\mathcal{N}}

\def\cH{\mathcal{H}}
\def\cS{\mathcal{S}}

\def\Sym{\mathrm{Sym}}

\def\cL{\mathfrak{L}}
\def\del{\partial}

\def\fh{\mathfrak{h}}
\def\ft{\mathfrak{t}}
\def\ff{\mathfrak{f}}
\def\fg{\mathfrak{g}}
\def\fq{\mathfrak{q}}

\def\fz{\mathfrak{z}}
\def\fp{\mathfrak{p}}

\def\cV{\mathcal{V}}
\def\RR{\mathbb{R}}
\def\LL{\mathbb{L}}

\def\Harm{\mathbf{H}}
\def\CC{\mathbb{C}}

\def\cX{\mathcal{X}}

\def\cU{\mathcal{U}}

\def\fF{\mathfrak{F}}
\def\fFc{\mathfrak{F}^c}
\def\la{\langle}
\def\ra{\rangle}

\def\d{\partial}

\def\n{\nabla}
\def\>{\rangle}

\def\<{\langle}
\def\>{\rangle}

\begin{document}

\title[Deformation of extremal metrics, complex manifolds]
{Deformation of extremal metrics, complex manifolds and the relative Futaki invariant}
\author[Y. Rollin]{Yann Rollin}
\author[S.R. Simanca]{Santiago R. Simanca}
\author[C. Tipler]{Carl Tipler}
\address{D\'epartment de Math\'ematiques, Laboratoire Jean Leray,
2, Rue de la Houssini\`ere - BP 92208, F-44322 Nantes, FRANCE}
\email{yann.rollin@univ-nantes.fr, srsimanca@gmail.com, 
carl.tipler@univ-nantes.fr}

\begin{abstract}
Let $(\mc{X},\Omega)$ be a closed polarized complex manifold, $g$ be 
an extremal metric on $\cX$ that represents the K\"ahler class $\Omega$, and
$G$ be a compact connected subgroup of the isometry group
$\Isom(\cX,g)$. Assume that the Futaki invariant relative to $G$ is 
nondegenerate at $g$. Consider a smooth family $(\mc{M}\to B)$ of polarized 
complex deformations of $(\cX,\Omega)\simeq (\cM_0,\Theta_0)$ provided with 
a holomorphic action of $G$ which is trivial on $B$. Then for every $t\in B$ sufficiently small,
there exists an $h^{1,1}(\cX)$-dimensional family of extremal
K\"ahler metrics on $\cM_t$ whose K\"ahler classes are arbitrarily 
close to $\Theta_t$. We apply this deformation theory to show that certain
complex deformations of the 
Mukai-Umemura $3$-fold admit  K\"ahler-Einstein metrics.
\end{abstract}

\maketitle

\section{ Introduction}
\label{secintro}
Every closed K\"ahler manifold is stable under complex deformations. Indeed, 
the classical result of Kodaira and Spencer \cite{ks} allows us to follow 
differentiably the K\"ahler metric under small perturbations of the complex
structure. Our goal here is the study of the stability of the {\it extremal} 
condition of a K\"ahler metric under complex deformations.

Let us recall  that a cohomology class
$\Omega\in H^2(\mc{X},\mb{R})$ on a complex manifold $(\mc{X},J)$ 
is called a {\it polarization} of $\mc{X}$ if $\Omega$ can be 
be represented by the K\"ahler form $\omega_g$ of a K\"ahler metric $g$
on $\mc{X}$. In this case, the pair $(\mc{X},\Omega)$ is called a polarized
complex  manifold. In what follows, we shall sometimes use the metric $g$ 
and its K\"ahler form $\omega_g$ interchangeably. We shall 
denote by $c_1=c_1(\mc{X},J)$ the first Chern class of $(\mc{X},J)$.

The set of all 
K\"ahler forms on $\cX$ representing a given polarization $\Omega$ is
denoted by ${\mathfrak M}_{\Omega}$. The search for a canonical representative
$g$ of $\Omega$ is done by looking for a critical point of the 
functional 
\begin{equation}\label{func}
\begin{array}{ccc}
\mathfrak{M}_{\Omega} & \rightarrow  & \mb{R} \\
 g  & \mapsto & {\displaystyle \int_\cX s_g^2 d\mu_g} 
\end{array}\, . 
\end{equation}
Here, $s_g$ is the scalar curvature of $g$ and
$d\mu_g$ its volume form. These critical points are the extremal metrics of 
Calabi \cite{c1}. The condition for $g$ to be such can be stated simply
by saying that the gradient of $s_g$ is a real holomorphic vector field, which
in itself shows a subtle interplay between extremal metrics and the 
complex geometry of $\mc{X}$.

We denote by ${\rm Aut}(\mc{X})$ 
the automorphism group of $\cX$. The space $\fh(\cX)$ of holomorphic vector
fields on $\cX$ is its Lie algebra. The space $\fh_0(\mc{X})$ of holomorphic 
vector fields with zeroes is an ideal of $\fh(\cX)$. 
The identity component $G'={\rm Isom}_0(\cX,g)$ of the isometry group of 
the metric metric $g$ is identified with a compact subgroup of 
${\rm Aut}(\mc{X})$, and its Lie algebra
is denoted $\fg'$. If $g$ is extremal, then
$G'$ is a {\it maximal connected compact subgroup} of
${\rm Aut}(\mc{X})$ \cite{c2}.

\subsection{Main result}\label{s11}
Let us consider a polarized complex manifold $(\cX,\Omega)$ and a smooth
family of complex deformations  $\mc{M}\to B$ of 
$\mc{X} \simeq \mc{M}_0$. Here $B$ is some open neighborhood of the origin 
in an Euclidean space $\mb{R}^m$ for some $m\geq 0$ 
(see \S\ref{deformations} for definitions). We let 
$E\to B$ be the vector bundle of second fiber cohomology of the fibration
$\mc{M}\to B$, whose fibers are $E_t=H^2(\mc{M}_t,\mb{R})$.
A smooth section $\Theta$ of $E\to B$ such that $\Theta_t$
admits a K\"ahler representative $\omega_t$ in $\mc{M}_t$ for all $t$ in $B$
is said to be a polarization of the family $\mc{M}\to B$. We obtain
a family $(\mc{M}_t,\Theta_t)$ of polarized manifolds parametrized by
$t\in B$. A polarized family of complex manifolds $(\mc{M}\to B,\Theta)$, and
a polarized complex manifold $(\mc{X},\Omega)$ together with an isomorphism 
$\mc{X} \simeq \mc{M}_0$ that makes $\Omega$ and $\Theta_0$ correspond to
each other, is said to be a polarized deformation of $(\mc{X},\Omega)$.

By shrinking the neighborhood of the origin $B$ if necessary, the 
Kodaira-Spencer theory allows us to choose a smooth $2$-form $\beta$ in
$\mc{M}$ such that $\omega_t=\beta|_{\mc{M}_t}$ is a K\"ahler form in
$\cM_t$ representing $\Theta_t$.
In this expression, $\beta|_{\mc{M}_t}$ 
denotes the pullback of $\beta$ by the canonical inclusion $\mc{M}_t
\hookrightarrow\mc{M}$. Such a $\beta$ is said to 
represent the polarization $\Theta$. If $g$ is a K\"ahler metric on $\mc{X}$ 
whose K\"ahler form $\omega_g$ represents $\Omega$, $\beta$ can be
constructed so that $\omega_0$ and $\omega_g$ agree under the isomorphism 
$\mc{X}\simeq \cM_0$. 
\medskip

Let us now assume that the metric $g$ is extremal. It is then a natural 
to ask if the representative $\beta$ of the polarization $\Theta$ can be chosen
so that the metric $g_t$, of K\"ahler form $\omega_t$ on $\mc{M}_t$, 
is extremal.
\medskip

A positive answer to this general statement is not to be
expected, and as it has been pointed out in \cite{ls2}, 
``the answer is an emphatic {\sl{no}}'' (cf. \cite{bdb} for some 
counter-examples). However, if we assume some symmetries for $\cM$ and the
nondegeneracy of the relative Futaki invariant, the answer is actually 
{\it yes} provided we allow the polarization $\Theta$ to be deformed
also.
\medskip

If the metric $g$ on $\cX$ is extremal, then $G'={\rm Isom}_0(\cX,g)$ is a
maximal connected compact subgroup of ${\rm Aut}(\cX,g)$ that
acts holomorphically on the central fiber $\cM_0$ of $\cM\to
B$. In general though, this action will not extend as a 
{\it holomorphic action} 
of $G'$ (cf. \S\ref{sec:holac} for a precise definition of this terminology) 
to the total space of deformations $\cM\to B$. We shall assume that the 
action of $G'$ extends partially, and we have a connected compact subgroup 
$G$ of $G'$ that acts holomorphically on $\cM\to B$ with trivial
action on $B$.
We denote by $C^\infty_G(\cM)$ the space of $G$-invariant 
smooth functions on $\cM$.

It is then possible to introduce the notion of
reduced scalar curvature $s^G_g$ for any $G$-invariant K\"ahler metric on $\cX$
(cf. \S\ref{sec:reducedscal}), and the Futaki
invariant relative to $G$ (cf. \S\ref{sec:relfut}) \cite{gbook}
$$
\fFc_{G,\Omega}:\fq/\fg\to \RR\, ,
$$
where $\fg$ is the Lie algebra of $G$ and $\fq$ is the normalizer of
$\fg$ in $\fh(\cX)$. The extremal condition of a metric is encoded in the 
equation $s^G_g=0$, and a $G$-invariant extremal metric representing 
$\Omega$ have vanishing reduced scalar curvature if, and only if,
$\fFc_{G,\Omega}$ vanishes identically \cite{fut,c2}. 
This characterization of extremal
metrics reinterprets in this manner a presentation advocated elsewhere
\cite{si0,si3,lusa}.

Let  $\fp$ be the normalizer of $\fg$ in $\fg'$. We set 
$\fg_0=\fg \cap \fh_0(\cX)$ and $\fp_0 =\fp\cap \fh_0(\cX)$, respectively.
The differential of the relative Futaki invariant induces a linear map
$$
\fq/\fg\to (H^{1,1}(\cX)\cap H^2(\cX,\RR))^*\, .
$$
We say that $\fFc_{G,\Omega}$ is nondegenerate if the
restriction of this map to $\fp_0/\fg_0$ is injective. 

Using the holomorphic action of $G$ on $\cM \to B$, we refine the 
construction of the representative $\beta$ of the polarizatin $\Theta$ so 
that the K\"ahler metric $g_t$ it induces on $\cM_t$ is
$G$-invariant (cf. \S\ref{eqsetup}) for all $t\in B$. 
We then use the corresponding K\"ahler form 
$\omega_t$ as the origin of the affine space of 
K\"ahler metrics on $\mc{M}_t$ that represent the polarization $\Theta_t$.
Let $\Harm_t$ be the space of real $g_t$-harmonic $(1,1)$-forms 
on $\cM_t$. By the Kodaira-Spencer theory, and perhaps
shrinking $B$ to a sufficiently small neighborhood of the origin, the
spaces $\Harm_t$ are the fibers of a smooth vector bunde $\Harm_t\to B$.
Given a section $\alpha$ of the bundle
$\Harm\to B$ and a sufficiently small function $\phi\in C^\infty_G(\cM)$,
we consider the family of K\"ahler metrics 
$g_{t,\alpha,\phi}$ on $\cM_t$ defined by the K\"ahler forms 
$$
\omega_{t,\alpha,\phi}= \omega_t + \alpha|_{\cM_t} + dd^c \phi|_{\cM_t}.
$$
Assuming that $g_0=g$ is extremal, we seek solutions 
$g_{t,\alpha, \phi}$ of the equation 
$$
s^G_{g_{t,\alpha,\phi}}=0 
$$
for values of $(t,\alpha,\phi)$ in a sufficiently small neighborhood of 
$(0,0,0)$ where the solution takes on the value $g_{0,0,0}=g$. 
Introducing suitable Banach spaces, the said equation 
defines as Fredholm map that is a submersion at the origin if, and only if,
the relative Futaki invariant is nondegenerate. The 
implicit function theorem then yields our main result:

\begin{theointro}
\label{theo} 
Let $(\mc{M}\to B,\Theta)$ be a polarized family of deformations 
 of a closed polarized complex manifold $(\mc{X},\Omega)$.
Suppose that $g$ is an extremal metric whose K\"ahler form $\omega_g$ 
respresents $\Omega$, and that $G$ is a
compact connected subgroup of $G'={\rm Isom}_0(\cX,g)$ such that 
\begin{itemize}
\item $G$ acts holomorphically on $\cM\to B$ and trivially on $B$,
\item the reduced scalar curvature $s^G_g$ of $g$ is zero,
\item the Futaki invariant relative to $G$ is nondegenerate at $g$.
\end{itemize}
Then, given any $G$-invariant representative $\beta$ of the polarization
$\Theta$ such that induced metric $g_0$ on $\cM_0$ agrees with $g$ via 
the isomorphism $\cM_0\simeq \cX$, and shrinking $B$ to a 
sufficiently small neighborhood of the origin if necessary, 
the space of K\"ahler metrics  on $\cM_t$ with vanishing reduced scalar 
curvature lying sufficiently close to the metric $g_t$ induced by 
$\beta$ is a smooth manifold of dimension  $h^{1,1}(\cX)$.
In particular, there are arbitrarily small perturbations $\Theta'$ of the
polarization $\Theta$ such that $\Theta'_t$ is represented by an
extremal metric. 
\end{theointro}
\begin{remark}
Theorem~\ref{theotech}  is a more precise and technical version of
Theorem~\ref{theo}. In particular, it contains more informations about
the kind of deformations $\Theta'$ of the polarization $\Theta$. This
refined version of the above theorem will lead to the applications below.
\end{remark}
\subsection {Applications}
Theorem \ref{theo} generalizes the results of \cite{ls,ls2}. The
case $G=\{1\}$ corresponds to the deformation theory of constant
scalar curvature K\"ahler metrics whereas the case $G=G'={\rm Isom}_0(\cX,g)$
corresponds to the deformation theory of extremal metrics
(cf. \S\ref{sec:appli} for details). The cases corresponding to 
intermediate choices for the group $G$ is new.
We illustrate the power of Theorem \ref{theo} with the analysis of 
some new examples of extremal metrics, and some cases of interest. However, we
observe that the applicability of our theorem to these
follows by elementary reasons, and the nondegeneracy of the relative Futaki
invariant is easy to see. It would be of interest to find applications
of our result in more demanding situations, which could 
 illustrate the problem at hand in further detail. Some highly nontrivial 
applications were already given in \cite{ls2}.

\subsubsection{Maximal torus symmetry}
Extremal K\"ahler metrics  are automatically stable
under complex  deformations with maximal torus symmetry:

\begin{corintro}\label{co2}
Let $(\mc{M} \to B,\Theta)$ be a polarized family of  deformations of 
a closed polarized complex manifold $(\cX,\Omega)$. Assume that
$\Omega$ admits an extremal representative and that $\cM\to B$ is
endowed with a holomorphic action of a maximal compact torus
$G=T\subset \Aut(\cX)$ acting trivially on $B$. 
 Then for $t\in B$ sufficiently small, $\Theta_t$ is represented
 by an extremal K\"ahler metrics on $\cM_t$.
\end{corintro}

Corollary \ref{co2} may look reminiscent of \cite[Lemma 4]{acgt}, where 
a stability result of the extremal condition under complex deformations 
with symmetries is obtained as an extension of the theory of \cite{ls}.
This result is carried out under such restrictive assumptions on 
the deformation of the complex structure and the K\"ahler class that 
the scope of its applicability is rather limited.

\subsubsection{The Mukai-Umemura $3$-fold}
We may apply Theorem \ref{theo} to the study of 
the Mukai-Umemura Fano $3$-fold $\cX$, with automorphism group
${\rm Aut}(\cX)=\PSL(2,\CC)$. Donaldson has proven that this variety 
admits a K\"ahler-Einstein metric \cite{Don08}. Our deformation 
theorem applies. We obtain the following result:

\begin{corintro}
\label{cor:mu1}
Let  $(\cM\to B,\Theta)$ be a polarized deformation of the
Mukai-Umemura $3$-fold with polarization 
$(\cX,c_1(\cX))$ and $\Theta_t=c_1(\cM_t)$. Assume $\cM\to B$ is one of the deformations
described at \S\ref{mudefo}, endowed with  a holomorphic
action of a group $G\subset 
\PSL(2,\CC)$
isomorphic to a dihedral group of order $8$ or a semidirect
product  $\SS^1\rtimes \ZZ/2$. 
 Then for every $t\in B$ 
sufficiently
small, $\Theta_t$ is represented by a K\"ahler-Einstein
metric on $\cM_t$.
\end{corintro}

The result above was proved already in \cite{Don08} using a different
approach which was later refined by Sz\'ekelyhidi~\cite{sz}. Using
this method it is possible to understand general small complex
deformations of the Mukai-Umemura $3$-fold. It turns out that small deformations corresponding to
polystable orbits under the action of $\PSL(2,\CC)$ are exactly the
one which carry Kaehler-Einstein metrics.

\subsection{Plan of the paper}
In \S \ref{s2} we recall further details and relevant facts about extremal 
metrics and real holomorphy potentials and relative Futaki invariant 
that we shall used throughout our
work. We proceed in \S \ref{deformations} to provide detailed 
definitions of complex and polarized deformations in a way suitable to
our work. The deformation problem that we treat by the implicit function 
theorem is presented in \S \ref{s4}, where we prove an expanded version of our
main Theorem \ref{theo}. The applications are given in \S\ref{sec:appli}. We 
begin
in \S \ref{lst} comparing our result to the LeBrun-Simanca deformation theory, 
and put the new contributions in proper perspective. We then study
deformations with maximal torus symmetry in \S \ref{tde}, where we 
prove Corollary \ref{co2}.
The Mukai-Umemura manifold and its deformations are discussed in \S \ref{mu3},
where we prove Corollary \ref{cor:mu1}.

\subsection{Acknowledgments}
We would like to thank Professor Simon Donaldson for pointing out an
error in the attempt to apply our deformation theory to the
Mukai-Umemura $3$-fold in an earlier version of this work. We also thank
Professor Paul Gauduchon for making the 
preliminary version of his book \cite{gbook} available to us.

\section{Extremal metrics}\label{s2}
Let $(\mc{X},\Omega)$ be a closed polarized complex manifold of 
dimension $n$. The polarizing assumption on the class $\Omega$ may be stated 
by saying that there exist K\"ahler metrics $g$ whose K\"ahler form 
$\omega_g$ represents it. We let ${\mathfrak M}_{\Omega}$ be the set of all 
such metrics. If $g\in {\mathfrak M}_{\Omega}$,
the projection of the scalar curvature $s_g$ onto the constants is given by
\begin{equation}
s_\Omega=4n\pi\frac{c_1 \cup \Omega^{n-1}}{\Omega^n}\, ,\label{s0}
\end{equation}
and there exists a vector field $X_{\Omega}\in {\mathfrak h}_0\subset
{\mathfrak h}$ \cite{fut,fm} such that for every $g\in {\mathfrak
  M}_{\Omega}$ we have
\begin{equation}
\int_\cX s_g^2 d\mu_g \geq s_{\Omega}^2\frac{\Omega^n}{n!}-
{\mathfrak F}(X_{\Omega},\Omega)\, . \label{lb}
\end{equation}
The equality is achieved if, and only if, the metric is extremal.
This lower bound was known earlier \cite{hw,si} for metrics in
${\mathfrak M}_{\Omega}$ that are invariant under a maximal compact subgroup 
of ${\rm Aut}(\mc{X})$. It was proven by Chen \cite{ch} to hold in general,
for any metric in ${\mathfrak M}_{\Omega}$. The lower bound varies smoothly
as a function of the polarizing class $\Omega$ \cite{si}.

\subsection{Holomorphic vector fields}
The subset ${\mathfrak h}_0$ of holomorphic vector 
fields with zeroes is an ideal of ${\mathfrak h}$, and the quotient 
algebra ${\mathfrak h}/{\mathfrak h}_0$ is Abelian.
A smooth complex-valued function $f$ gives rise to the $(1,0)$-vector field 
$f\mapsto \partial^{\#}_g f$ defined by the expression
$$
g(\d_g^{\#} f, \hspace{1mm}\cdot \hspace{1mm})=\db f \, .
$$
This vector field is holomorphic if, and only if
$\over{\partial}\partial_g^{\#} f=0$, a condition equivalent 
to $f$ being in the kernel of the \emph{Lichnerowicz operator}
\begin{equation}
L_g f:= (\over{\d}{\d}_g^{\#})^{\ast}\over{\d}{\d}_g^{\#}f= 
\frac{1}{4}\Delta_g^2 f + 
\frac{1}{2}r ^{\mu\nu}\n_{\mu}\n_{\nu}f+ \frac{1}{2}(\n^{\over{\ell}}s_g )
\n_{\over{\ell}}f \, .\label{lic}
\end{equation}

The ideal ${\mathfrak h}_0$ consists of vector fields of the
form $\partial^{\#}_g f$, for a function $f$ in the kernel of $L_g$. Or put
differently, a holomorphic vector field $\Xi$ can be written
as $\partial_g^{\#}f$ if, and only if,
the zero set of $\Xi$ is nonempty \cite{ls2}. The kernel $\mc{H}_g$ of
$L_g$ is called the {\it space of holomorphy potentials} of $(\cX,g)$.
Since $L_g$ is elliptic, $\cH_g$ is finite dimensional complex vector
space and consists of
smooth functions.

Generally speaking, the Lichnerowicz operator $L_g$ of a metric $g$ is not
a real operator, and so the real and imaginary parts of a function in its 
kernel do not have to be elements of the kernel also. In studying the geometry 
of K\"ahler manifolds, it is often convenient to use real quantities and 
operators. We introduce the terminology that allows us to do so here.
We follow the conventions of \cite{be,gbook}.

Let us consider the operator $d^c$ on functions defined by $d^c=Jd$.
Given a real holomorphic vector field $\Xi$ on $\cX$, the  
Hodge theory provides a unique decomposition 
$\xi=\Xi^\flat=\xi_{h}+d u_\Xi + d^c v_\Xi$, where $\xi_{h}$ is a harmonic 
$1$-form,  $u_\Xi$, $v_\Xi$ are
real valued functions, and $d^c v_\Xi$ is coclosed. The functions 
$u_\Xi$, $v_\Xi$ are uniquely determined up to an additive constant.
The dual of this identity gives rise to the decomposition
$\Xi=\Xi_{h}+ \grad u_\Xi + J \grad v_\Xi$, and it follows that 
$\Xi^{1,0}=\frac 12(\Xi-iJ\Xi) =\Xi^{1,0}_{h}+\del_g^\# f_\Xi$, where
$f_\Xi=u_\Xi +iv_\Xi$. It follows that $\Xi$ is real holomorphic if, 
and only if, $\Xi^{1,0}$ is complex holomorphic, that is to say, if, and 
only if, $L_gf_\Xi=0$. If we extend the Lie bracket operation linearly in each 
component, the integrability of the complex structure
makes of the map $\Xi \rightarrow \Xi^{1,0}$ an isomorphism of Lie algebras.

As indicated above, the condition $\Xi_{h}=0$ is 
equivalent to $\Xi$ being Hamiltonian, and characterizes $\fh_0(\cX)$. 
If $\Xi \in \fh_0(\mc{X})$, the potential $f_\Xi$ of $\Xi$ 
is a solution of the equation $\cL_{\Xi^{1,0}}\omega_g = \frac 12 dd^c f$. 
On the other
hand, a holomorphic vector field $\Xi$ is parallel if, and only if, $\Xi$
is the dual of a harmonic $1$-form.
   
A vector field $\Xi$ is a Killing field if, and only if, it is holomorphic 
and its potential function $f_\Xi=iv_{\Xi}$ is, up to a constant, a purely 
imaginary function \cite{ls2}. In that case, $v_\Xi$ will be called the 
Killing potential of $\Xi$. 

If $\xi$ is a $1$-form, we let $\nabla^-\xi$ be the $J$ anti-invariant 
component of the $2$-tensor $\nabla^g \xi$. Let $\xi= \Xi^\flat$ for a real
vector field $\Xi$. The condition $\nabla^-\xi = 0$ is equivalent to 
$\Xi$ being a real holomorphic vector field, and can be expressed as
$\delta\delta\nabla^-\xi =0$.

We have the identities 
\begin{equation}
  \label{eq:ddnablater}
   \delta\nabla^-\xi= \frac 12 \Delta_g \xi - J\xi^\sharp\hok \rho_g
\end{equation}
and 
\begin{equation}  \label{eq:ddnabla}
  \delta\delta\nabla^-\xi=\frac 12 \Delta_g \delta\xi - \< d^c\xi,\rho_g \> +
  \frac 12\< \xi,ds_g \>\, ,
\end{equation}
respectively.

We introduce the operator $\LL_g$ defined by 
$$
\LL_g f=  (\nabla^-d)^*(\nabla^-d) f = \delta\delta \nabla^-df\, .
$$
This is a real elliptic operator of order four, and if $f$ is a real valued 
function, we have that 
$\LL_g f=0$ if, and only if, $\grad f$ is a 
real holomorphic vector field; by the observations above, every 
Hamiltonian Killing field is of the form $J\grad f$ for a function $f$ of 
this type. 

By (\ref{eq:ddnabla}) we derive the identities
\begin{equation}
  \label{eq:ddnablabis}
\begin{array}{rcl}
   \LL_g f & = &  \frac 12 \Delta_g^2 f+\la dd^c f,\rho_g\ra +
\frac 12 \la df,ds_g \ra\, , \vspace{1mm} \\
 \delta\delta\nabla^-d^c f & = & - \frac 12
  \cL_{J\grad s_g} f\, , 
\end{array}
\end{equation}
and by (\ref{lic}), we see that 
\begin{equation}
  \label{eq:decomp}
2L_g f= \LL_g f + \frac i2 \cL_{J\grad s_g}f\, .
\end{equation}

\begin{lemma}
For any K\"ahler metric $g$, the space of real solutions of the equation
$\LL_g f=0$ coincide with the space of real solutions of $L_g f=0$.
\end{lemma}

\begin{proof}
By (\ref{eq:decomp}), if $f$ is a real valued function 
such that $L_g f=0$, then $\LL_g f= 0$ and 
$\cL_{J\grad s_g}f=0$. Conversely, let us assume that $\LL_g f=0$. This 
implies that $\Xi=J\grad f $ is a Killing field. Then
$$
\cL_{J\grad s_g} f = \la d^cs_g,df\ra= -\la ds_g , d^c f\ra = -\cL_\Xi 
s_g =0\, , 
$$
and by (\ref{eq:decomp}), we conclude that $L_g f=0$.
\end{proof}

\subsection{The group of isometries of a K\"ahler metric}
\label{sec:kgi}
Let $(\cX,\Omega)$ be a polarized complex manifold and $G$, $G'$ be
connected compact Lie subgroups of ${\rm Aut}(\cX)$, with $G\subset G'$. 
By taking a $G'$-average if necessary, we represent the polarization 
$\Omega$ the K\"ahler form $\omega_g$ of a $G'$-invariant K\"ahler metric
$g$. 
We attach to this data a relative Futaki invariant.

\subsubsection{Lie algebras}
\label{sec:liealg}
The Lie algebras $\fg$ and $\fg'$ of $G$ and $G'$ are naturally identified 
with subalgebras of the algebra $\fh(\cX)$ of holomorphic vector fields. 
We define $\fg_0=\fg\cap\fh_0(\cX)$ and $\fg'_0=\fg' \cap \fh_0(\cX)$.
We then introduce the Lie algebras 
\begin{itemize}
\item $\fz=Z(\fg)$, the center of $\fg$,
\item  $\fz'=C_{\fg'}(\fg)$, the centralizer of
$\fg$ in $\fg'$, 
\item  $\fz''=C_{\fh}(\fg)$, the centralizer of
$\fg$ in $\fh$, 
\item  $\fp=N_{\fg'}(\fg)$, the normalizer of $\fg$ in $\fg'$, and
\item  $\fq=N_{\fh}(\fg)$, the normalizer of $\fg$ in $\fh$. 
\end{itemize}
If $\ft$ is any of these Lie algebras, we shall denote by $\ft_0$
the ideal of Hamiltonian vector fields $\ft_0=\ft\cap\fh_0(\cX)$, and by
$\cH_g^{\ft_0}$ the corresponding spaces of holomorphy
potentials.

The space of holomorphy potentials $\cH^{\fz_0}_g$ (respectively
 $\cH^{\fz'_0}_g$, $\cH^{\fz''_0}_g$) is identified to the
$G$-invariant potentials of $\cH^{\fg_0}_g$ (respectively
 $\cH^{\fg_0'}_g$, $\cH_g$). Notice that $\cH^{\fz_0}_g \subset
 \cH^{\fg_0}_g$ and $\cH^{\fz'_0}_g \subset \cH^{\fg'_0}_g$ 
consists of purely imaginary functions whose imaginary parts define
Killing potentials.

Now $\fz$ is an ideal of $\fz'$ and $\fz''$. On the other hand, $\fg$
is an ideal of $\fp$ and $\fq$. We have canonical injections
\begin{equation}
  \label{eq:caninc}
\fz'/\fz \hookrightarrow \fp/\fg, \quad    \fz'_0/\fz_0 \hookrightarrow 
\fp_0/\fg_0,\quad
\fz''/\fz \hookrightarrow \fq/\fg, \quad    \fz''_0/\fz_0 \hookrightarrow 
\fq_0/\fg_0\, .
\end{equation}

\begin{lemma}
\label{lemma:isomlie}
The canonical injections \eqref{eq:caninc} are surjective, and we 
have canonical isomorphisms of Lie algebras
$$
\fz'/\fz \simeq \fp/\fg,\quad \fz'_0/\fz_0 \simeq \fp_0/\fg_0, \quad 
\fz''/\fz \simeq \fq/\fg,\quad \fz''_0/\fz_0 \simeq \fq_0/\fg_0\, .
$$
\end{lemma}

\begin{proof}
 We prove that $\fp=\fz'+ \fg$. 
Let $\Xi$ be a Killing field in $\fg'$. We have that
$$
\Xi= \Xi_{h} + J \grad v_{\Xi} 
$$
where $\Xi_{h}$ is the dual vector field of a harmonic $1$-form, and 
$v_{\Xi}$ is a real function. The vector field $\Xi$ belongs to $\fp$ if,
and only if,
 \begin{equation}
   \label{eq:normal}
[Y,\Xi]=\cL_Y\Xi \in \fg \quad\text{for all $Y\in \fg$.}   
 \end{equation}
Since $\cL_Y\Xi_{h}=0$ because $Y$ is Killing, this condition
is equivalent to $\cL_YJ\grad v_{\Xi} \in \fg$. In turn, this is equivalent to 
having $J\grad (Y\cdot v_{\Xi}) \in \fg$ since $Y$ preserves $J$ 
and the metric, which means that $Y\cdot v_{\Xi} \in i\cH_g^{\fg_0}$
for every $Y\in \fg$. This implies that for every $\gamma\in G$, we have
$\gamma^*v_{\Xi}- v_{\Xi} \in i\cH_g^{\fg_0}$. We can then 
average the function $v_{\Xi}$ under the group
action to obtain a $G$-invariant function $\tilde{v}_{\Xi}$ on $\cX$ such
that $\hat v= v_{\Xi}- \tilde{v}_{\Xi} \in i\cH_g^{\fg_0}$, and
$$
\Xi= (\Xi_{h} +  J \grad \tilde{v}_{\Xi}) + J\grad \hat v\, ,
$$
where $J\grad\hat v \in \fg$ and $ \Xi_{h} + J\grad \tilde{v}_{\Xi} =
\Xi -  J \grad \hat v \in \fg'$ is $G$-invariant and so an element of
$\fz'$.

An analogous argument shows that $\fq = \fz''+\fg$, which finishes 
the proof. 
\end{proof}

\subsubsection{The reduced scalar curvature}
\label{sec:reducedscal}
Let $L^{2}_{k}(\cX)$ be the $k$th Sobolev space defined by $g$. The space 
$L^2_{k,G}(\cX)$ of $G$-invariants functions in 
$L^{2}_{k}(\cX)$ can be obtained as metric completion of $C^\infty_{G}(\cX)$.
If $k>n$, $L^{2}_{k}(\cX)$ is a Banach algebra. In fact, the
Sobolev embedding theorem says that if $k>n+l$ we have  
$L^{2}_{k}(\cX)$ continuously contained in $C^{l}(\cX)$, the space of 
functions with continuous derivatives of order at most $l$. We shall work 
below always imposing this restriction over $k$.

The $L^2$-Hermitian product  on $L^2_{k,G}(\cX)$ induced  by the Riemannian
metric $g$ allows us to define the orthogonal
$W_{k,g}$ of $i\cH^{\fz_0}_g$. Thus we have an orthogonal decomposition
$$
L^2_{k,G}(\cX)= i\cH^{\fz_0}_g \oplus  W_{k,g}
$$
together with the associated projections 
$$
\pi^W_g: L^2_{k,G}(\cX) \to  W_{k,g} \quad
\mbox { and }\quad \pi^G_g: L^2_{k,G}(\cX) \to  i\cH^{\fz_0}_{g}.
$$
We introduce the \emph{reduced scalar curvature} $s^G_g$
defined by
$$
s^G_g=\pi^W_g(s_{g}) =
(\BOne -\pi^G_g)(s_{g})
$$
By construction, the condition
\begin{equation}
s^G_g=0  
\end{equation}
is equivalent to $s_g\in i\cH^{\fz_0}_g$, and since 
$\cH^{\fz_0}_g\subset \cH_g$, this condition implies the extremality of
 the  metric $g$.

\subsubsection{The reduced Ricci form}
Let $L^2\Lambda^{1,1}_{k,G}(\cX)$ be the space of real
$G$-invariant $(1,1)$-forms on $\mc{X}$ in $L^2_{k}$. We lift the projection
$\pi^G_g$ to the a projection 
$$
\Pi^G_g: L^2\Lambda^{1,1}_{k+2,G}(\cX)\to L^2\Lambda^{1,1}_{k,G}(\cX)
$$
defined by
$$
\Pi^G_g\beta = \beta +  dd^c f \, ,
$$
where $f= G_g(\pi^W_g(\omega_g,\beta))$ and $G_g$ is the Green operator of $g$
\cite{si2, si0, lusa}. 
It follows that
$(\omega_g, \Pi^G_g(\beta)) = \pi^G_g (\omega_g ,\beta)$.

We apply this projection map to the Ricci form, and define the reduced
Ricci form \cite{si0}
$$
\rho^G := \Pi^G\rho\, .
$$
We then obtain the identity
\begin{equation}
  \label{eq:redricci}
\rho_g = \rho^G + \frac 12 dd^c \psi_g^G,  
\end{equation}
where $\psi_g^G=2G_g(\pi_g^W(\omega_g, \rho))= G_g\pi_g^W(s_g)=G_g(s^G_g)$. 
In particular, $\rho^G=\rho_g$ if, and only if, $s_g^G=0$.

\subsubsection{Variational formulas}
Given a K\"ahler metric $g$, we consider infinitesimal 
deformations of it given by 
$$
\omega_{t}=\omega_g+t\alpha +t dd^c\phi
$$
of the K\"ahler class $\omega_g$, where $\phi$ is a smooth function and
$\alpha$ is a $g$-harmonic $(1,1)$-form, respectively. These variations
have K\"ahler classes given by $[\omega_g]+t[\alpha]$. In order to avoid
confusion, the derivative at $g$ of various geometric quantities 
will be denoted by $D_g$ below. Sometime later, and when confusion cannot 
occur, these same derivatives will be denoted by a dot superimposed on the 
quantities themselves.

The holomorphy potential of a holomorphic vector field depends upon
the choice of the metric. Its infinitesimal variation when moving the metric
in the direction of $\phi$ or $\alpha$ is described in
the following lemma:

\begin{lemma}
\label{lemma:varpot}
Let $\Xi$ be a real holomorphic vector field in $\fh_0(\cX)$, with holomorphy
potential $f_\Xi=u_\Xi+iv_\Xi\in\cH_g$. The variation $D_g f_\Xi 
(\phi)$ of $f_{\Xi}$ at $g$ in the direction of $\phi$ is given by
$$
D_{g}f_\Xi (\phi )= 2\cL_{\Xi^{1,0}}f_\Xi\, ,
$$
which is equivalent to the expressions $(D_g u_\Xi )(\phi)=\cL_\Xi \phi$ 
and $(D_g v_\Xi )(\phi)=-\cL_{J\Xi} \phi$ for the variations of the real
and imaginary parts of $f_{\Xi}$. The variations $D_g u_\Xi (\alpha)$ and
$D_g v_\Xi (\alpha)$ at $g$ in the direction of a trace-free form $\alpha$ are 
given by 
$$
\begin{array}{rcl}
D_g u_\Xi (\alpha) & = &  -G_g(\delta(J\Xi\hok \alpha))\, , \\
D_g v_\Xi (\alpha) & = & G_g(\delta(\Xi\hok \alpha))\, .
\end{array}
$$
\end{lemma}

The variation of the reduced scalar curvature $s^G_g$ is described by:

\begin{lemma}
\label{lemma:diff1}
Let $g$ be a $G$-invariant K\"ahler metric on $\cX$.
The variation of the reduced scalar curvature $s^G_g$ when moving the metric 
in the direction of $\phi$ is given by
$$
(D_g s^G_g)(\phi)= - 2 \LL_g \phi + \la d\phi,ds_g^G\ra\, .
$$
If $s^G_g=0$, the variation of $s^G_g$ when varying the metric in the 
direction of a trace-free form $\alpha$ is given by
$$
(D_g s^G_g)(\alpha)= \pi^W(G_g(\la \alpha,dd^c s_g\ra) -
2\la \alpha,\rho_g \ra).
$$
\end{lemma}

\subsection{The relative Futaki invariant}
\label{sec:relfut}
Given any $G$-invariant K\"ahler metric $g$ on $\cX$ with K\"ahler form
$\omega_g$ and Ricci-form $\rho_g$, we define the Futaki function of
$(G,g)$ by 
$$
\fFc_{G,\omega_g}(\Xi) = \int_{\cX} d^c \psi_g^G(\Xi)\; d\mu_g = 
\int_{\cX} -\cL_{J\Xi} \psi_g^G \; d\mu_g
$$
for any real holomorphic vector field $\Xi$ on $\cX$. Here, $\psi_g^G$
is the Ricci potential (\ref{eq:redricci}) of the metric $g$.
This function is defined in terms the K\"ahler metric $g$, but 
depends only upon the K\"ahler class $\Omega=[\omega_g]$ \cite{fut, c2}, as
we briefly recall below. The resulting function $\fFc_{G,\Omega}$ shall
be referred to as the Futaki invariant of $\Omega$ relative to $G$, or 
relative Futaki invariant for short, when the group $G$ and class $\Omega$
are understood.
Let us observe that the usual definition of the real Futaki invariant
applied to  
a holomorphic vector field $\Xi$ is given by
$\fF_{G,\Omega}(\Xi)=\fFc_{G,\Omega} (J\Xi)$. We have 
introduced this notation for convenience. Put differently, there is a
complex valued version of the Futaki invariant such that $\fF_{G,\Omega}$ is its
real part and $\fF^c_{G,\Omega}$ the imaginary part.
\medskip

\subsubsection{Properties of the Futaki invariant}
The function $\fFc_{G,\omega_g}(\Xi)$ vanishes on parallel holomorphic vector
fields. Indeed, let $\Xi_{h}$ be the dual of a harmonic $1$-form $\xi_{h}$. 
Then 
$$
\fFc_{G,\omega}(\Xi_{h})=\int_{\cX}\la J\xi_{h},
d\psi^G_g\ra\; d\mu_g= \int_{\cX}\la \delta J\xi_{h},
\psi^G_g\ra\; d\mu_g\, .
$$
Since the space of harmonic $1$-forms is $J$-invariant, so $\delta J\xi_{h}=0$,
and it follows that $\fFc_{G,\omega_g}(\Xi_{h})=0$. 

This function $\fFc_{G,\omega_g}$ can be expressed alternatively in
terms of the reduced scalar curvature. For if $\Xi$ is a Hamiltonian 
holomorphic 
vector field on $(\cX,g)$, it can be written as 
$\Xi=\grad u_\Xi + J\grad v_\Xi$ for some real valued functions $u_\Xi, 
v_\Xi$, and with $d^c u_\Xi$ orthogonal to the space of closed forms.
Hence, 
$\fFc_{G,\omega}(\Xi) = \la d v_\Xi, d\psi^G_g\ra =
\la v_\Xi, \delta d\psi^G_g\ra  =\la v_\Xi, \Delta_g \psi^G_g\ra $. It
follows that
\begin{equation}
\fFc_{G,\omega_g}(\Xi) = \fFc_{G,\omega_g}(\grad u_{\Xi}+J\grad v_\Xi)=
\int_\cX v_\Xi s^G_g \, d\mu_g \, . \label{eq:futscal}  
\end{equation}

The expression above shows that if $\Xi$ is any Killing field in $\fg$, then
$\fFc_{G,\omega}(\Xi)=0$.  It follows that $\fFc_{G,\omega_g}$ 
induces a $\RR$-linear map
\begin{equation}
  \label{eq:relfu}
\fFc_{G,\omega_g} : \fq/\fg \to \RR \, ,  
\end{equation}
 where $\fq$ is the normalizer of $\fg$ in $\fh(\cX)$. 

The fundamental properties of this function is now stated in the form of 
a Theorem. This was proven originally by Futaki \cite{fut}, and 
generalized a bit later by Calabi \cite{c2}. The version given here follows
its presentation in \cite{gbook}.

\begin{theorem}
The relative Futaki function $\fFc_{G,\omega_g}:\fq/\fg\to \RR$ defined 
in {\rm (\ref{eq:relfu})} is independent of the particular  
$G$-invariant K\"ahler representative $\omega_g$ of the class
$\Omega$, and so it induces an invariant function
$$
\fFc_{G,\Omega}:\fq/\fg\to \RR
$$
of the class, the relative Futaki invariant of $\Omega$. 
The relative Futaki invariant vanishes for a class $\Omega$
if, and only if, any $G$-invariant extremal metric $g$ that
represents $\Omega$ has vanishing reduced scalar curvature $s^G_g$.
If $g$ is any $G$-invariant extremal K\"ahler metric and
$G'={\rm Isom}_0(\cX,g)$, the vanishing of $\fFc_{G,\Omega}$ is
equivalent to the vanishing of its restriction 
$$
\fFc_{G,\Omega}: \fp_0/\fg_0\to\RR\, .
$$
\end{theorem}

{\it Proof}. The invariance is proven in \cite{fut}, and generalized in
\cite[Proposition 4.1, p. 110]{c2}.

  The proof of the second statement follows by
(\ref{eq:futscal}). For if $s^G_g=0$, then $\fFc_{G,\omega}(\Xi)=0$ for every
$\Xi\in\fq_0$. Conversely, let us assume that $\fFc_{G,\Omega}$ vanishes on
$\fq$ and, therefore, on $\fq_0=\fq\cap \fh_0(\cX)$. Let $g$ be a
$G$-invariant extremal K\"ahler metric that represents the class $\Omega$,
 and let $G'={\rm Isom}_0(\cX,g)$. Then $\Xi=J\grad s_g \in \fg'_0$, and by
\cite[Theorem 1, p. 97]{c2}, $\cL_\Xi Z=0$ for all Killing fields $Z$. Thus, 
we have that $\Xi\in \fp_0\subset\fp\subset \fq$.
But
$$
\fFc_{G,\Omega}(\Xi)=0=\int_\cX s_gs^G_g\, d\mu_g =\|s^G_g\|^2_g\, ,
$$ 
and so $s^G_g=0$.

The final statement follows easily by the argument above.
\qed

\subsubsection{Nondegeneracy of the Futaki invariant}
Let $g$ be a metric such that $[\omega_g]=\Omega$. Given any
real $g$-harmonic $(1,1)$-form, we may use the metric $g$ to compute
the derivative  
$$
\frac{d}{dt} \fFc_{G,\Omega+t\alpha}\mid_{t=0}\, .
$$ 
We have the following:

\begin{lemma}
\label{lemma:derfut}
Let $g$ be a $G$-invariant extremal K\"ahler metric on $\mc{X}$ 
such that $[\omega_g]=\Omega$, and let $\alpha$ be a real 
$g$-harmonic trace-free $(1,1)$-form. If 
$\Xi=\grad u_\Xi + J \grad v_\Xi$ is a holomorphic vector field in $\fq_0$, 
then
\begin{align*}
\frac{d}{dt} \fFc_{G,\Omega+t\alpha}\mid_{t=0}(\Xi) &= 
\int _{\cX} v_\Xi (D_g s^G_g)(\alpha)\, d\mu_g \\
 &=
\int_\cX \pi^W(v_\Xi)  \left (G_g(\la \alpha, 
dd^c s_g \ra)- 2\la \alpha,\rho_g \ra \right )\, d\mu_g \, .  
\end{align*}
\end{lemma}

{\it Proof}. Since $\alpha$ is trace-free, the infinitesimal variation of the
volume form is zero. We use this fact in the differentiation of the 
relative Futaki invariant given by expression (\ref{eq:futscal}), and obtain 
that
$$
\frac{d}{dt} \fFc_{G,\Omega+t\alpha}\mid_{t=0}(\Xi)= 
\int_{\cX} \dot v_\Xi(\alpha) s^G_g \, d\mu_g +
\int_{\cX} v_\Xi 
\dot s^G_g(\alpha)\, d\mu_g =\int_{\cX} v_\Xi 
\dot s^G_g(\alpha)\, d\mu_g 
$$
where the last equality follows because the reduced scalar curvature $s_g^G$ 
of $g$ vanishes. 
We then use Lemma \ref{lemma:diff1} to obtain that
$$
\frac{d}{dt} \fFc_{G,\Omega+t\alpha}\mid_{t=0}(\Xi)= 
\int_\cX \pi^W (v_\Xi) (G_g(\la \alpha, dd^c s_g
\ra) - 2\la \alpha,\rho_g\ra)\, ,
$$
which finishes the proof. 
\qed 

Notice that we can state the result above using $\rho_g^G$ instead because
$\rho_g = \rho_g^G$ for metric $g$ with vanishing reduced scalar curvature.

The differential of the relative Futaki invariant of 
Lemma \ref{lemma:derfut} defines a linear mapping 
\begin{equation}
 \label{eq:fund}
\fq/\fg \to (H^{1,1}(\cX)\cap H^2(\cX,\RR))^*\, .  
\end{equation}
This leads to a very important concept in our work:

\begin{definition}
Let $G$ and $G'$ be connected compact subgroups of ${\rm Aut}(\mc{X})$
such that $G\subset G'$.
The Futaki invariant $\fFc_{G,\Omega}$ relative to $G$  is said to be
$G'$-nondegenerate at $\Omega$ if the linear map 
\eqref{eq:fund} restricted to 
$$
\fp_0/\fg_0\simeq \fz'_0/\fz_0  \to (H^{1,1}(\cX)\cap H^2(\cX,\RR))^*
$$
is injective. If $g$ is a K\"ahler metric  representing
$\Omega$ such that $G'={\rm
  Isom}_0(\mc{X},g)$ and  this condition holds for some $G\subset G'$, 
we say that $g$ is Futaki nondegenerate relative to $G$.
\end{definition}

We briefly illustrate this notion in a particular case next.
 
\subsubsection{Nondegeneracy of the Futaki invariant relative
  to a maximal compact torus}
In the presence of certain maximal torus symmetries on the underlying manifold, the 
relative Futaki invariant vanishes. 

\begin{lemma}
\label{torus}
Let $\mc{X}$ be a complex manifold, and let $T$ be a maximal compact
torus subgroup of ${\rm Aut}(\mc{X})$. 
Then for $G=T$, we have $\fq/\fg=0$ and the relative Futaki invariant
$\fFc_{G,\Omega}$ is identically zero for any K\"ahler class $\Omega$ in 
$\cX$. Given any compact connected Lie group $G' \subset {\rm Aut}(\mc{X})$,
then for $G=T \subset G'$ the relative Futaki invariant $\fFc_{G,\Omega}$ is
$G'$-nondegenerate. 
\end{lemma}

\begin{proof}
The Lie algebra $\fg$ of $G=T$ is Abelian and so equal to its
center $\fz$. The centralizer $\fz''$ of $\fg$ in $\fh$ contains $\fz$. 
If $\fz$ were strictly contained in $\fz''$, we could find an element of 
$\fz''\setminus \fz$ that together with $\fz$ would generate an Abelian Lie 
algebra, and this would contradict the
maximality of $\fz$ in $\fh$. Thus, $\fz=\fz''$, and by 
Lemma \ref{lemma:isomlie}, $\fq/\fg\simeq \fz''/\fz=0$. 
The result follows.
\end{proof}

\section{Deformations of the complex structure}
\label{deformations}

In this section we recall the theory of smooth deformations of a complex
manifold and that of smooth polarized deformations.
We illustrate the concept exhibiting a deformation of the Hirzebruch surface
$\mc{F}_1$. This deformation will reappear in one of our applications in
\S \ref{sec:appli}.

\subsection{Complex deformations}
 A smooth family of complex deformations consists of the following data:
\begin{enumerate}
\item  an open connected 
 neighborhood $B$ of the origin in $\mb{R}^m$, 
a smooth
manifold $\cM$ and a smooth proper submersion  $\varphi :
\mc{M}\rightarrow B$,
\item  an open covering $\{\mc{U}_j\}_{j\in I}$ of $\mc{M}$ and 
smooth complex valued functions $z_j=(z^1_j,\cdots,z^n_j)$ defined on 
each $\mc{U}_j$, such that the collection of mappings
$$
\begin{array}{ccl}
\mc{U}_j \cap \varphi^{-1}(t) &\to & \CC^n\\
p & \mapsto &(z^1_j(p), \ldots, z_j^n(p))  
\end{array}
$$
define a holomorphic atlas on each manifold $\varphi^{-1}(t)$.

\end{enumerate}
Such a family of complex deformations will be denoted simply by
$\mc{M} \to B$, and the fibers $\varphi^{-1}(t)$ together with their
canonical complex structure by $\cM_t$. A complex manifold  $\cX$ and
a family of deformations $\cM\to B$ 
together with a given isomorphism $\cX\simeq\cM_0$ is called a complex
deformation of $\cX$.
\medskip

Smooth families of complex deformations are smoothly locally trivial. Indeed,
 let 
$M$ be the underlying differentiable manifold of the central fiber $\mc{M}_0$
of a family $\mc{M}\to B$.
At the expense of shrinking $B$ if necessary, we 
can find a diffeomorphism $\mc{M} \rightarrow B\times M$ 
such that the diagram 
$$
\xymatrix{\mc{M}  \ar[rd] \ar[rr] && B\times M \ar[ld]^{\pi_{B}} \\
& B&}
$$
commutes. Here $\pi_B$ denotes the projection map onto the first factor. 
We refer to this diffeomorphism as a trivialization of the given family of
complex deformations. By way of such a trivialization, we see that 
all the $M_t$s are diffeomorphic to $M$ for $t$s that are in a
sufficiently small neighborhood of the origin in $\mb{R}^m$, 
and that the family of complex manifolds $\mc{M}_t$ can be seen as
a differentiable family $\{ J_t\}$ of integrable almost 
complex structures on $M$. From this point of view, $\mc{M}_t$ and $(M,J_t)$
are identified as complex manifolds. 

\subsection{Polarized deformations}
Let us now assume that $(\mc{X},\Omega)$ is a polarized complex manifold, and
that $(\mc{M}\to B,\Theta)$ is a polarized complex deformation of it.
If $\omega_g$ is a K\"ahler form on $\mc{X}$ that represents $\Omega$, a deep 
result of Kodaira and Spencer \cite{ks} 
shows that we can find a smooth family of K\"ahler
metrics on $\mc{M}_t$ that extends $\omega_g$ and represents $\Theta_t$ for 
each $t$ (cf. \S \ref{s11}). For in the language of integrable almost complex 
structures $\{ J_t\}$ on $M$ that a trivialization of the deformation allows, 
the function 
$t\rightarrow h^{p,q}_t ={\rm dim}_{\mb{C}} H^{q}(M ,\Omega^p_{J_t})$ is upper
semi-continuous and if $(M, J ,\omega_g)$ is K\"ahler, this 
function is in fact constant in a sufficiently small neighborhood of the 
origin. Then it follows that there exists a smooth family 
$t\rightarrow \omega_t$ of $2$-forms on $M$ 
such that $\omega_t$ is K\"ahler with respect to $J_t$, 
$[\omega_t]=\Theta_t\in H^2(M; \mb{R})$, and $\omega_0$ and $\omega_g$ 
agree with each other via the identification 
$\mc{X}\simeq \mc{M}_0\simeq (M,J_0)$. 
This point of view is best adapted to our work here.  We obtain 
a $2$-form $\beta$ on $\mc{M}$ such that $\beta\mid_{\mc{M}_t}=\omega_t$.
Such a form $\beta$ is said to represent the polarization $\Theta$.

\subsection{Example: the Mukai-Umemura $3$-fold}\label{mudefo}
Let $V$ be a $7$-dimensional complex vector space,
$Gr_3(V)$ be the Grassmanian of complex $3$-dimensional 
subspaces of $V$  and  $U\rightarrow Gr_3(V)$ be the
tautological rank $3$-bundle over $Gr_3(V)$. Notice that 
$Gr_3(V)$ is $12$-dimensional. 

Any $\varpi \in \Lambda^2V^*$ defines an section $\sigma_{\varpi}$
of the bundle $\Lambda^2U^*\to Gr_3(V)$. Let $Z_{\varpi}\subset
Gr_3(V)$ be the zero set of $\sigma_{\varpi}$. So 
$Z_\varpi$ is the
subset of isotropic $3$-planes of $\varpi$, the points $P$ in 
$Gr_3(V)$ such that $\varpi\mid_{P}=0$. For a generic $\varpi$, $Z_{\varpi}$ 
is a smooth subvariety of codimension three, and given three linearly 
independent forms,  $\varpi_1$, $\varpi_2$, $\varpi_3$, we obtain the $3$-fold 
$$
X_{\Pi}=X_{\varpi_1,\varpi_2, \varpi_3}=Z_{\varpi_1}\cap Z_{\varpi_2}
\cap Z_{\varpi_3} \subset Gr_3(V) \, ,
$$
that depends only on the $3$-plane $\Pi$  spanned by the $\varpi_i$s in 
$\Lambda^2V^*$, and not in the basis chosen 
to represent it. 
The action of the group $\mb{S}\mb{L}(V)$ 
on $V$ induces an action  on $Gr_3(\Lambda^2V^*)$, and $3$-planes 
$\Pi_1$ and $\Pi_2$ define isomorphic complex varieties
$X_{\Pi_1}$ and $X_{\Pi_2}$ if, and only if, the planes $\Pi_1$ and $\Pi_2$ 
lie in the same $\mb{S}\mb{L}(V)$ orbit. 
We obtain a set of equivalence 
classes of $3$-folds parametrized by the quotient $\mc{U}/\mb{S}\mb{L}(V)$. 

There is a Zariski open set $C\subset Gr_3(\Lambda^2V^*)$ of 
$3$-planes $\Pi$ such that $X_{\Pi}$ is a smooth
subvariety of dimension $3$. It has an obvious family of complex
deformations. For if $\cN=\{(\Pi,
x)\in C\times Gr_3(V)|\; x\in X_\Pi\}$, we may consider this smooth
complex manifold together with its canonical projection
$\cN\to C$. 
The $\SL(V)$-action on $C\times Gr_3(V)$ induces an equivariant action
on $\cN\to C$, and two fibers are isomorphic if, and only if,
they are above the same orbit in $C$.

The Mukai-Umemura manifold is a particular smooth $3$-fold in this
family. It can be described
efficiently as follows. The six symmetric power ${\rm Sym}^6
(\mb{C^2})$ is the standard irreducible $7$-dimensional representation
of $\SL(2,\CC)$.
We take $V={\rm Sym}^6 (\mb{C^2})$ with its induced
$\SL(2,\CC)$-action. The representation $\Lambda^2 V^*$  
decomposes into irreducible representations as
$$
\Lambda^2 ({\rm Sym}^6 (\mb{C}^2))= {\rm Sym}^{10}(\mb{C}^2)\oplus 
{\rm Sym}^6 (\mb{C}^2)\oplus {\rm Sym}^2 (\mb{C}^2)\, ,
$$
The summand ${\rm Sym}^2(\mb{C}^2)$ corresponds to a $3$-plane $\Pi_0$
in $\Lambda^2 V^*$ that defines the Mukai-Umemura variety $X_{\Pi_0}$. 
The plane $\Pi_0$ is invariant under 
$\mb{S}\mb{L}(2,\mb{C})$, so this group acts naturally on $X_{\Pi_0}$. 

Since many of the deformations are equivalent via the $\SL(V)$-action, it
is important to describe the quotient of $Gr_3(\Lambda^2 V^*)$. This
is done carefully in \cite{Don08}, where it is proven that the quotient
of the tangent space to $Gr_3(\Lambda^2V^*)$ at $\Pi_0$ by the tangent
space to the the orbit can be identified with $\Sym^8(\CC^2)$ with its
standard $\SL(2,\CC)$-action, the stabilizer of $\Pi_0$ in $\SL(V)$.
By the theory for equivariant slices of Lie group actions, there is an
neighborhood of the origin $B\subset \Sym^8(\CC^2)$ and a
$\PSL(2,\CC)$-equivariant embedding $j:B\to Gr_3(\Lambda^2 V^*)$ such
that for $t_1,t_2\in B$, the images $j(t_1)$ and $j(t_2)$ are in the
same $\SL(V)$-orbit if, and only if, $t_1$ and $t_2$ are in the same
$\PSL(2,\CC)$ orbit. If $B$ is taken to be sufficiently small, we have
$j(B)\subset C$ and $\cM\to B$, defined as the fiber product $\cM
=B\times_C \cN$, is a smooth family of complex deformations of
$\cM_0\simeq X_{\Pi_0}$.

\subsubsection{Deformations with symmetries}
In particular, the deformations  corresponding to
polynomials $p=C(u^4-\alpha v^4)(v^4-\alpha u^4)$ for $\alpha,C\in
\CC^\times$  have a stabilizer $G$ in $\PSL(2,\CC)$ isomorphic to a
dihedral group of order $8$. In the case where $\alpha=0$, the
stabilizer of $p$ is the subgroup of $\PSL(2,\CC)$ spanned by the one
parameter subgroup $\lambda \cdot [u:v]=[\lambda u:\lambda^{-1}v]$ and
the rotation $[u:v]\mapsto [u:v]$. Hence, we have 
a maximal compact subgroup $G$ of the
stabilizer of $p$ that is isomorphic to $\ZZ/2\rtimes\SS^1$.

We can consider the family of deformation of $X_{\Pi_0}$ obtained by
restricting $B$ to the subspace of polynomials of the form $tp$ as
above for $t\in \CC$. Thus, we get a family of deformation $\cM\to \CC$
endowed with a holomorphic action of $G$, where $G$ is a dihedral
group of order $8$, or the semidirect product of $\ZZ/2\rtimes\SS^1$ in
the case where $\alpha=0$. In addition, the group $G$ acts on the
central fiber as a subgroup of $\PSL(2,\CC)$, the identity component
of the automorphism group of $X_{\Pi_0}$. Therefore, $G$ acts
trivially on the cohomology of every fiber $\cM_t$ of the deformation
$\cM\to \CC$.

\section{Deformations of extremal metrics}\label{s4}
In this section we prove a criterion that ensures the stability of
the extremal condition of a K\"ahler metric under complex deformations.
We assume some symmetries of the family of deformations and the nondegeneracy 
of the relative Futaki invariant.

Let $\cM\to B$ be a smooth family of complex deformations of a complex
manifold $\cX\simeq \cM_0$. If we assume that $\cX$ is of K\"ahler type, 
then it follows that all fibers $\cM_t$ are K\"ahler provided we shrink the set
of parameters $t\in B$ to a sufficiently small neighborhood of the 
origin \cite{ks}. In particular, the Lie algebra of holomorphic vector fields
$\fh(\cM_t)$ contains the ideal $\fh_0(\cM_t)$ consisting of
holomorphic vector fields with zeroes somewhere (cf. \S\ref{s2}).
Throughout this section, we shall always assume that $\cX$ is of
K\"ahler type and that $B$ has been so restricted. We shall indicate
the occasions where the latter restriction may be necessary.

\subsection{Holomorphic group actions}
\label{sec:holac}
Let $G$ be a compact connected Lie group acting smoothly on $\cM$ 
such that:
\begin{itemize}
\item The fibers $\cM_t$ are preserved under the action.
\item The induced action on each $\cM_t$ is holomorphic,
\item $G$ acts faithfully on $\cM_0$, and it is identified 
to a subgroup of the connected component of the identity of
${\rm Aut}(\cM_0)$.
\end{itemize}
Under this conditions, we say that  $G$ {\it acts holomorphically on}
$\cM$ and trivially on $B$.

As discussed in \S \ref{deformations}, we may think of the deformation 
$\cM\to B$ as
a smoothly varying family of integrable almost complex structure $J_t$ on the 
underlying manifold $M$ of the central fiber such that $(M,J_t)\simeq \cM_t$,
and using the smooth trivialisation of the deformation near the central fiber, 
the holomorphic $G$-action on $\cM$ can be seen as a smoothly varying 
family of $G$-actions 
$$
\begin{array}{llll}
  a:&B\times G\times M&\to& M\\
&(t,g,x)&\mapsto& a_t(g,x) 
\end{array}
$$
where $a_t$ is a $J_t$-holomorphic action that is identified to the
$G$-action on $\cM_t$ modulo the isomorphism  $(M,J_t)\simeq \cM_t$.

By \cite{ps}, we know that 
compact Lie group actions are rigid up to conjugation. Hence, 
there exists a smooth isotopy $f_t:M\to M$ that intertwines the $a_t$ and 
$a_0$ actions,
$$
\text{$f_t^{-1}\circ a_t(g,f_t(x))= a_0(g,x)$ for all $t\in B$, $g\in G$,
$x\in M$,}
$$
possibly after restricting $B$ to some smaller neighborhood of the origin.
So acting by $f_t$ on the family of
complex structures $J_t$, we may assume that the action of $G$ is
independent of $t$, and that it is holomorphic relative to each $J_t$.

The triviality of the smooth deformation of the action of
$G$ on $M$ has  some strong consequences on the complex geometry. 
Firstly, there is a canonical morphism of the Lie algebra $\fg$ of
the group $G$ into the space of smooth vector fields on $M$,
$$
\xi: \fg \hookrightarrow  C^{\infty}(M,TM)\, .
$$
a map that is injective because the action of $G$ is assumed to be
faithfull on the central fiber. We may therefore think of $\fg$ as a 
subalgebra of $C^{\infty}(M)$. We will do so, and drop $\xi$ from the notation
when no confusion can arise. Notice also that since the $G$-action is $J_t$
holomorphic, $\fg$ is a subalgebra of $\fh(M,J_t)$ for all $t\in B$.

We consider the ideal $\fh_0(M,J_t)$ of Hamiltonian holomorphic vector fields
in $\fh(M,J_t)$, the space of holomorphic vector fields with a nontrivial zero
set \cite{ls}. Then the ideal of $\fg$ given by
$\fg_0 = \fg \cap \fh_0(M,J_t)$  consists of the
vector fields in $\fg\subset C^{\infty}(M,TM)$ that vanish somewhere, and since
this properties is independent of the complex structure $J_t$, $\fg_0$ is
independent of $t$.

We summarize the discussion above into the following proposition:

\begin{proposition}
\label{prop:holac}
Let  $\cM\to
B$ be a family of complex deformations of a manifold $\cX$ of K\"ahler
type. Let $G$ be a compact connected Lie group  acting
holomorphically on $\cM$ and trivially on $B$, and let $\fg$ be its Lie algebra.
If $B$ is restricted to a sufficiently small neighborhood of the origin,
there exists a smooth trivialization $\cM_t\simeq (M,J_t)$ such that
$\cM_t$ is K\"ahler for all $t\in B$, the action of $G$ on $\mc{M}_t$ is
independent of $t$, the image of the natural 
embedding $\fg\subset C^{\infty}(M,TM)$ is contained  in $\fh(M,J_t)$ for
all $t\in B$, and $\fg_0=\fg\cap \fh_0(M,J_t)$ is an ideal of $\fg$ that 
is independent of $t$.
\end{proposition}

\begin{definition}
Let $\mc{M}\to B$ be a family of complex deformations of a manifold $\cX$
of K\"ahler type, and let $G$ be a compact 
connected Lie group acting holomorphically on $\mc{M}$ and trivially
on $B$. A smooth 
trivialisation $\cM_t\simeq (M,J_t)$ that satisfies the
properties of Proposition \ref{prop:holac} is said to be
an {\it adapted trivialization} for $\cM\to B$ relative to $G$.
\end{definition}

\subsection{The equivariant deformation problem}
\label{eqsetup}
Let  $(\cM\to B,\Theta)$ be a polarized family of complex deformations
of a polarized manifold $(\cX,\Omega)$, and let $G$ be a compact connected 
Lie group acting holomorphically on $\cM$ and trivially on $B$. 
Since $G$ is connected it acts trivially on the
cohomology of the fibers $\cM_t$, and in particular on $\Theta_t$.
Let $(M,J_t)\simeq \cM_t$ be an adapted trivialization for $\cM\to B,\Theta)$
relative to $G$. With some abuse of notation, we denote by $\Theta_t$ the
polarization induced on $M$ via the isomorphism $\cM_t\simeq (M,J_t)$.
Then we have the following:

\begin{lemma}
\label{lemma:adaptedmet}
Let $(\cM\to B,\Theta)$ be a polarized deformation of 
$(\cX,\Omega)\simeq (\cM_0,\Theta_0)$ provided with a holomorphic action of 
a compact connected Lie group $G$ acting trivially on $B$. Consider an adapted trivialization 
$(M,J_t)\simeq \cM_t$, and let $g$ be an extremal metric on $\cX$ that 
represents the K\"ahler class $\Omega$. 
Then, if $B$ is restricted to a sufficiently small neighborhood of the origin,
there exists a smooth family of $G$-invariant K\"ahler metrics $g_t$ on 
$(M,J_t)$ that represent the K\"ahler classes $\Theta_t$, and 
such that $g_0$ is identified to the metric $g$ by the isomorphism 
$\cX\simeq \cM_0$ up to conjugation by an element of ${\rm Aut}(\cX)$.
\end{lemma}

\begin{proof}
The connected component of the identity in the group of
iso\-metries denoted $\Isom_0(\cX,g)$ is a maximal compact connected subgroup of
${\rm Aut}(\cX)$ \cite{c2}. Thus, we may assume that $G\subset
\Isom_0(\cX,g)$ up to conjugation by an element of ${\rm Aut}(\cX)$.

Let $g_0$ be the $G$-invariant metric on $(M,J_0)$ that corresponds to $g$ by
the isomorphism $\cX\simeq (M,J_0)$. By the Kodaira-Spencer theory, we can
extend $g_0$ to a smooth family of K\"ahler metrics $g_t$ on $(M,J_t)$ that
represent $\Theta_t$ for $t$s sufficiently small. We can 
average these metrics if necessary to make of them
$G$-invariant. Since $G$ acts isometrically on $g_0$, the averaging process 
leaves $g_0$ unchanged. On the other hand, $G$ acts trivially on the
cohomology, and hence, on $\Theta_t$. So the averaging of the metric $g_t$ 
does not change the K\"ahler class that the metric represents. 
This finishes the proof.
\end{proof}

\begin{definition}
Let $(\mc{M}\to B,\Theta)$ be a polarized family of deformations of 
$(\cX,\Omega)$ provided with a holomorphic action of a compact connected 
Lie group $G$ acting trivially on $B$. Assume that $\Omega$ is represented by an extremal metric $g$.
A smooth family of $G$-invariant K\"ahler metrics $g_t$ 
satisfying the properties given by Lemma \ref{lemma:adaptedmet} for an 
adapted trivialization is said to be an 
{\it adapted smooth family of K\"ahler metrics}.
\end{definition}
\medskip

Let us notice that for the adapted family of metrics $g_t$ of
Lemma \ref{lemma:adaptedmet}, the metric $g_0$ coincides with the metric
$g$ on $\mc{X}$ up to the conjugation by an automorphism. 
Throughout the rest of the paper, we shall assume
that the isomorphism $\cX\simeq\cM_0$ has been so conjugated so
that $g_0$ and $g$ coincide. 

In our considerations below, the group $G'$ of \S\ref{sec:kgi} will always
be the connected component of the identity in the 
isometry group ${\rm Isom}(\mc{X},g)$ of $g$.

\subsection{Analytical considerations}
Let $(\cM\to B,\Theta)$ be a polarized deformation of $(\cX,\Omega)$,
and let $G$ be a compact Lie group acting holomorphically on $\cM\to
B$ and trivially on $B$.
We assume that $G$ is contained in $G'$. 
Let $g$ be an extremal metric on $\cX$ that represents the class
$\Omega$. We consider an adapted trivialization $\cM_t\simeq (M,J_t)$ and
adapted family of K\"ahler metrics $g_t$ on $(M,J_t)$
(cf. \S\ref{sec:holac} and \ref{eqsetup}) such that
$g_0=g$.

Let  $L^{2}_{k}(M)$ be the $k$th Sobolev space defined by $g$, and let
$L^2_{k,G}$ be space of elements in $L^2_k$ that are $G$-invariant. 
The latter can be defined as the Banach space 
completion of $C^\infty_{G}(M)$ in the $L^2_k$ norm.

We denote by $\omega_t$ the K\"ahler form of the metric $g_t$ on $(M,J_t)$ 
of the adapted family. 
Let $\Harm_t(M)$ be the space of  $g_t$-harmonic real $(1,1)$-forms on
$(M,J_t)$. By Hodge theory, we have that 
$$
\Harm_t(M)\simeq H^{1,1}(M,J_t))\cap H^2(M,\RR)\, .
$$
Since $G$ is connected, it acts trivially on the cohomology of $M$. By the
uniqueness of the harmonic representative of a close form, it therefore acts
trivially on $\Harm_t$ also. Further, the Kodaira-Spencer theory
shows that all the spaces $\Harm_t(M)$ are isomorphic for $t$s sufficiently
small, and form the fibers of a vector bundle $\Harm(M)\to B$.

For a given function 
$\phi \in L^2_{k+4,G}(M)$ and $\alpha\in \Harm_t(M)$ sufficiently small, one
can define a new K\"ahler metric  $g_{t,\alpha,\phi}$ on $(M,J_t)$ with K\"ahler form
$$\omega_{t,\alpha,\phi}=\omega_t + \alpha + dd^c\phi.
$$
By definition, the deformed metric $g_{t,\alpha,\phi}$ is
automatically invariant under the  $G$-action and
represents the K\"ahler class $\Theta_{t,\alpha}=\Theta_t+[\alpha]$.

The real and complex Lichnerowicz operator of $g_{t,\alpha,\phi}$ will be denoted
$\LL_{t,\alpha,\phi}$ and
 $L_{t,\alpha,\phi}$
 respectively. The space of Killing potentials for the metric
 $g_{t,\alpha,\phi}$ is given by
$i\ker \LL_{t,\alpha,\phi}$. 

Since $G$ acts 
isometrically on 
$g_{t,\alpha,\phi}$, the Lie algebra   $\fg_0\subset \fh_0(M,J_t)$
consists of Hamiltonian
Killing fields  for the metric
$g_{t,\alpha,\phi}$ as well.
Let $\cH^{\fg_0}_{t,\alpha,\phi}$ be the space of 
holomorphy potentials corresponding to the Killing fields in $\fg_0$ for
the metric $g_{t,\alpha,\phi}$. The space
$\cH^{\fg_0}_{t,\alpha,\phi}$ consists of purely imaginary functions
of the form $iv$ where $v$ is a Killing potential for some Killing
vector field in $\fg_0$ relative to the metric $g_{t,\alpha,\phi}$.
Using the notations of \S\ref{sec:liealg}, we  introduce 
$\cH^{\fz_0}_{t,\alpha,\phi}$ as the $G$-invariant part of
$\cH^{\fg_0}_{t,\alpha,\phi}$. 
An essential feature of $\cH^{\fz_0}_{t,\alpha,\phi}$ is that it is
identified to $\RR\oplus \fg_0^G$, where $\fg_0^G$ is the Lie
subalgebra of $Ad(G)$-invariant vector fields in $\fg_0$ (or equivalently
the $G$-invariant one when $\fg_0$ is considered as a Lie subalgebra
of $\ff(M)$).  

 It follows
that the spaces $\cH^{\fz_0}_{t,\alpha,\phi}$ have constant dimension and that
they are the fibers of a vector bundle $\cH^{\fz_0}$ over a
neighborhood of the origin in the total space of the bundle 
$L^2_{k+4,G}(M)\oplus \Harm(M) \to B$.

The $L^2$-norm on $L^2_{k',G}(M)$ induced  by the Riemannian
metric $g_{t,\alpha,\phi}$ allows us to define the orthogonal
$W_{k',t,\alpha,\phi}$ of $i\cH^{\fz_0}_{t,\alpha,\phi}$ and an orthogonal 
direct sum
$$
L^2_{k',G}(M)= i\cH^{\fz_0}_{t,\alpha,\phi}\oplus  W_{k',t,\alpha,\phi}
$$
varying smoothly with $(t,\alpha,\phi)$.
This construction provides Banach bundles  
$$W_{k'}\to \cV$$
 where
$\cV$ is a sufficiently small neighborhood of the origin in the
total space of $L^2_{k+4,G}(M)\oplus \Harm(M) \to B$.
We shall denote by 
$$
\pi^W_{t,\alpha,\phi}: L^2_{k',G}(M) \to  W_{k',t,\alpha,\phi}\quad
\mbox { and }\quad \pi^G_{t,\alpha,\phi}: L^2_{k',G}(M) \to  i\cH^{\fz_0}_{t,\alpha,\phi}
$$
the canonical projection associated to the above splitting.
The \emph{reduced scalar curvature} $s^G_{t,\alpha,\phi}$ of
$g_{t,\alpha,\phi}$ is given  by
$s^G_{t,\alpha,\phi}=\pi^W_{t,\alpha,\phi}(s_{g_{t,\alpha,\phi}}) =
(\BOne -\pi^G_{t,\alpha,\phi})(s_{g_{t,\alpha,\phi}})$.
We are looking for particular extremal metrics near $g$, namely the one with
vanishing reduced scalar curvature 
\begin{equation}
  \label{eq:secs}
s^G_{t,\alpha,\phi}=0 .
\end{equation}
The LHS of \eqref{eq:secs} can be interpreted as a section of the
bundle $W_k\to \cV$. Our goal is to seek the zeroes of
this section of Banach bundle.
Keeping a more prosaic style we can express
equation~\eqref{eq:secs} more concretely by using suitable
trivialisations of the 
relevant bundles as follows:
the metric $g_0$ induces an $L^2$-norm and an associated 
 orthogonal projection $P=\pi^W_0:L^2_{k+4,G}(M)\to  W_{k,0}$. For
 parameters $(t,\alpha,\phi)$ sufficiently small, the restriction 
$$
P:W_{k,t,\alpha,\phi}\to W_{k,0},
$$
is automatically an isomorphism.

The bundle of harmonic real $(1,1)$-forms  $\Harm(M)\to B$,
admits a trivialization in a neighborhood of the central fiber.
Thus we have a smooth
isomorphism of vector bundle $h$, up to the cost of shrinking $B$ to
some smaller neighborhood of the origin, which commutes with the canonical projections
$$
\xymatrix{
B\times \Harm_0(M)\ar[rr]^h\ar[dr]& &\Harm(M)\ar[dl]\\
&B &
}
$$
and such that $h$ restricted to the central fiber is the identity. We shall use the notation $h(t,\alpha)=h_t(\alpha)\in \Harm_t(M)$.

Let $\cU$ be a sufficiently small open neighborhood of the origin in
$B\times \Harm_0(M)\times W_{k+4,0}$ such that the following map is
 defined 
\begin{equation}
  \label{eq:mapSbis}
  \begin{array}{lclc}
    \cS:&\cU &\to &  B\times W_{k,0}\\
&(t,\alpha,\phi)&\mapsto &
\left (t,P\left (s^G_{t,h_t(\alpha),\phi}\right )\right )
  \end{array}
\end{equation}

\begin{lemma}
\label{diff}
  The map $\cS$ is $C^1$ and its differential is a Fredholm operator.
Assuming that the K\"ahler metric $g$ on $\cX$ has vanishing reduced scalar
curvature $s^G_g=0$, the differential at $(t,\alpha,\phi)=0$  is given by 
given by a linear operator of the form
$$
\left (\begin{array}{cc}
  \BOne & 0 \\
* & S_{G,g}
\end{array}\right )
$$
where 
$$
S_{G,g}(\dot\alpha,\dot \phi) = - 2 \LL_g\dot\phi +P(\dot s
^G_g(\dot\alpha))
$$
is the differential of $P(s^G)$ at $g$ 
in the direction of $(\dot\alpha,\dot\phi)$.
In the case where $\dot\alpha$ is tracefree, we have
$$
S_{G,g}(\dot{\alpha},0)=
\dot s^G_g(\dot\alpha)= P(G_g(\la \dot\alpha,dd^c s_g\ra) -
2\la\dot\alpha,\rho\ra)\, .
$$
\end{lemma}
\begin{proof}
The map $\cS$ is $C^1$ since the reduced scalar curvature depends in a
$C^1$ manner of the data $(t,\alpha,\phi)$. The computation of the
differential of $\cS$ is deduced from Lemma~\ref{lemma:diff1}.
\end{proof}

At this point, we may compute the index of the differential of
$\cS$:
\begin{lemma}
Under the assumption $s^G_g=0$, the index of the differential of 
$\cS$ at the origin is equal
  to $h^{1,1}(\cX)$.
\end{lemma}
\begin{proof}
  By Lemma \ref{diff}, the operator $S_{G,g}$
 is a compact perturbation of the map
$$
\begin{array}{cll}
  \Harm_0(M) \times W_{k+4,0}&\to &  
  W_{k,0}\\
(\alpha,\phi)&\mapsto & -2\LL_g\phi
\end{array}
$$
The Lichnerowicz operator $\LL_g: W_{k+4,0} \to W_{k,0}$ has index
$0$. We conclude that the differential of $\cS$ has index $\dim
\Harm_0(M) = h^{1,1}(\cX)$. 
\end{proof}

\subsection{Surjectivity}
We return to the study of the map $\cS$ with the notations of
\S\ref{eqsetup}. 
\begin{proposition}
\label{prop:subm}
Under the additional assumption that
  $g$ is an extremal metric on $\cX$ with  $s^G_g=0$, 
 the map $\cS$ defined at \eqref{eq:mapSbis} is a submersion at the
  origin if and only the relative Futaki invariant $\fFc_{G,\Omega}$ is non
  degenerate at $\Omega$, the K\"ahler class of $g$.
\end{proposition}
\begin{proof}
  The cokernel of the differential of $\cS$ is identified to $\psi\in
  W_{k,0}$ such that
$$
\la  \LL_g\dot \phi,\psi\ra =0,\quad \la  P(\dot s^G_g(\dot\alpha)),
\psi\ra =0
$$
for all $\dot\phi\in W_{k+4,0}$ and $\dot \alpha$ 
$g$-harmonic $(1,1)$-forms on $(M,J_0)\simeq \cX$.

The first equation implies that $\LL_g\psi = 0$. Therefore we have
$\psi\in i \cH^{\fz'_0}_g \simeq \RR\oplus \fz'_0$. Let $\Xi\in\fz'_0$ be the Killing field represented by
$\psi$. 
Using the second condition in the particular case where $\dot\alpha$
is tracefree, we see that
$$ 0= \la  P(\dot s^G_g(\dot\alpha)),
\psi\ra= \la  \dot s^G_g(\dot\alpha),
\psi\ra = \int_\cX \psi \dot s^G_g(\dot\alpha)\,  d\mu_g=\dot
\fFc_{G,\Xi,\Omega}(\dot\alpha),$$
where the last equality is given by Lemma~\ref{lemma:derfut}.
The relative Futaki nondegeneracy condition implies that $\Xi\in \fz_0$, in
other words $\psi\in i\cH^{\fz_0}_g$. By definition
 $\psi$ is orthogonal to $i\cH^{\fz_0}_g$, hence $\psi=0$ and $\cS$ is
 a submersion.

Conversely, it is easy to check that if the relative Futaki invariant
is degenerate, then $\cS$ is not a submersion.
\end{proof}

\begin{remark}
Under the Futaki nondegeneracy assumption, we may choose any linear
subspace $V\subset \Harm_0(M)$ such that the linearized Futaki
invariant induces an injective map $\fp_0/\fg_0\to V^*$. Then the
corresponding restriction of  $\cS$ is still a submersive map at the origin.  
\end{remark}

In conclusion, we obtain the following theorem:
\begin{theorem}
\label{theoB} \label{theotech}
Let $(\cM\to B,\Theta)$ be a  polarized family of complex deformations
of a polarized manifold $(\cX,\Omega)$. Assume that $\cM\to B$ is
endowed with a holomorphic action of a connected compact Lie group $G$
acting trivially on $B$
and that $\cX$  admits a $G$-invariant
extremal metric $g$ with K\"ahler class $\Omega$ and such that $s^G_g=0$.

Given an adapted trivialization  $\cM_t\simeq (M,J_t)$
defined for $t$ sufficiently small, let $g_t$ be any adapted smooth family of
$G$-invariant K\"ahler metrics on $(M,J_t)$ representing
$\Theta_t$ (cf. \S\ref{sec:holac} and \S\ref{eqsetup}).

Assume that  the
relative Futaki invariant $\fFc_{G,\Omega}$ is non
degenerate at $g$, then choose a space $V\subset \Harm_0(M)$ such that
the linearized relative Futaki invariant restricted to $\fp_0/\fg_0\to V^*$ is
injective.
Then, the space of solutions 
$$S=\{(t,\alpha,\phi)\in\cU \quad |\quad \alpha\in V \mbox{ and }s^G_{g_{t,h_t(\alpha),\phi}}=0\}$$
is a smooth
manifold of real dimension $\dim V+\dim B$, in a sufficiently
small neighborhood of the origin. For any $(t,\alpha,\phi)\in S$,  $\alpha$ and
$\phi$ are automatically smooth.

The canonical projection $S\to B$ is a submersion near the origin and 
the fibers are $\dim V$-dimensional submanifold of $S$
corresponding to families of
$G$-invariant K\"ahler metrics $g_{t,\alpha,\phi}$   on $\cM_t$ with
vanishing reduced scalar curvature  representing a perturbation of the
polarization given by $\Theta_t+[\alpha]$.
\end{theorem}
\begin{proof}
  The hypothesis imply that the map $\cS$ restricted to
  $\cU'=\{(t,\alpha,\phi)\in \cU\;\|\;  \alpha \in V\}$ is a submersion at the
  origin. The kernel $K$ of the differential of $\cS$ has dimension equal
  to its index $\dim V$. Let $\pi_K:W_{k+4}\to K$ be the
  orthogonal projection onto $K$. By definition, the map
$$
\cU'\stackrel{\pi_K\times\cS}\longrightarrow K\times B\times W_{k,0}
$$
is an isomorphism. The implicit function theorem provides the desired 
solutions parametrized by $B\times K$.
\end{proof}

\begin{remark}
  When $\fq/\fg=0$, the Futaki invariant is automatically non
  degenerate and Theorem~\ref{theotech} applies for any subspace
  $V\subset \Harm_0(M)$. Setting $V=0$, the theorem provides for every $t$
  sufficiently small a unique extremal metric  on $\cM_t$ with K\"ahler
  class $\Theta_t$. In other words, the original polarization $\Theta$ does not
  have to be perturbed in this case.
\end{remark}

\subsection{Generalization to nonconnected groups}
\label{nonconnected}
Although to this point the group $G$ has been assumed to be connected, this
assumption is not necessary to a large extent. The connectedness was used 
to ensure that $G$ acts trivially on the cohomology of the manifold, hence 
on harmonic forms. Thus, we merely need the assumption that $G$ acts 
trivially on on the relevant K\"ahler classes.

For instance, if  $G$ is contained in the connected
component of the identity of $\Aut(\cX)$, then $G$ acts
trivially on the cohomology of $\cX$. We can check that the
definition of the Futaki invariant relative to $G$ still makes
sense. The analysis developped at \S\ref{eqsetup} extends trivially in
this case by working $G$-equivariantly. The only difference in this more 
general framework will appear in Proposition \ref{prop:subm}. For in order 
to make sure that this proposition still holds, we should change slightly 
the definition of relative
Futaki nondegeneracy. Let us recall that the linearized relative Futaki
invariant induces a map $\fz'_0/\fz_0\to (H^{1,1}(\cX)\cap
H^2(\cX,\RR))^*$. If $G$ is connected, $\fz'_0$ agrees with  the space
of $G$-invariant Hamiltonian Killing 
fields on $(\cX,g)$. But if $G$ is not connected, there is a residual
action of $G$ on $\fz'_0$ by a 
finite group action. The space of  $G$-invariant vector fields in
$\fz'_0$ is denoted  $\fz'_{0,G}$. Similarly we denote by $\fz_{0,G}$
the $G$-invariant part of $\fz_0$. There is an embedding
$\fz'_{0,G}/\fz_{0,G} \subset  \fz'_{0}/\fz_{0}$ and we shall say that
the Futaki invariant $\fF^c_{\Omega,G}$ is nondegenerate at $g$ if the
map 
$$\fz'_{0,G}/\fz_{0,G}\to  (H^{1,1}(\cX)\cap
H^2(\cX,\RR))^*$$
 is injective.

Using this new definition of the Futaki nondegeneracy, we can drop the 
connectivity assumption on $G$  and just assume 
that $G\subset G'=\Isom_0(\cX,g)$ in Theorem \ref{theotech}. We can then
derive the same conclusions.

\section{Applications}
\label{sec:appli}
\label{s5}
In this section we apply Theorem \ref{theo} in some particular
situations, and especially to
produce new examples of K\"ahler manifolds with extremal metrics.

\subsection{Relation to LeBrun-Simanca deformation theory}\label{lst}
It is easy to see that Theorem \ref{theo} enables us to recover the
deformation theory of \cite{ls2}.

\subsubsection{Case $G=\{1\}$}  Let $g$ be a K\"ahler metric on $\cX$
with K\"ahler class $\Omega$. Then, we have $s^G_g=s_g-\bar s_g$, 
where $\bar s_g$
is the average  of $s_g$ on $\cX$. So the condition $s^G_g=0$ is
equivalent to the property that $g$ has constant scalar
curvature. Furthermore the Futaki invariant relative to $G=\{1\}$
agrees with the non-relative Futaki invariant. In this case,
Theorem~\ref{theo} is equivalent to the deformation theory for cscK
metrics of~\cite{ls2}.

\subsubsection{Case $G=G'$} If $g$ is extremal, then $G'=\Isom_0(\cX,g)$
is a maximal connected compact subgroup of $\Aut(\cX)$. Then every
extremal metric is $G$-invariant, upto conjugation, and the condition
$s^G_g=0$ for a $G$-invariant K\"ahler metric is equivalent to the
metric being extremal. 

Theorem~\ref{theo} applies. In the case of trivial complex
deformations with $B=\{0\}$, one recovers the openness theorem of
\cite{ls2}. We actually seem to have a more general result for it is possible
to allow complex deformations with a holomorphic action of~$G$ acting
trivially on $B$.

\subsection{Deformations with maximal torus symmetry}\label{tde}
The deformation theory under  a maximal compact torus symmetry 
is particularly
well behaved. 
Let $\cX$ be  complex manifold and $g$ an extremal metric on
$\cX$ with K\"ahler class $\Omega$. We shall denote by $G=T^n\subset
\Aut(\cX)$ a maximal  compact
torus.
 Up to conjugation by an automorphism, we may
assume that the metric $g$ is $G$-invariant so that $G\subset
G'=\Isom_0(\cX,g)$. The vector field
$\Xi=J\grad s_g$ is in  $\fz'\subset \fz''$. But in this case
$\fz''=\fz$ by maximality of $\fz$ (cf. proof of Lemma~\ref{torus}).
Hence $\Xi\in \fz\subset \fg$, which implies $s^G_g=0$.
The Futaki invariant is trivial, hence nondegenerate
(cf. Lemmas~\ref{torus}). So we may apply
Theorem~\ref{theo} to polarized deformations $(\cM\to B,\Theta)$ of
$(\cX,\Omega)$ endowed with a holomorphic action of $G$ acting
trivially on $B$.
 Thus we have proved Corollary~\ref{co2}.

\subsection{The Mukai-Umemura $3$-fold}\label{mu3}
Let $X_{\Pi_0}$ be the Mukai-Umemura $3$-fold of \S\ref{mudefo}. We use the
deformation $\cM\to \CC$ of $X_{\Pi_0}$ that we described there, which is
provided with an holomorphic action of a group $G$ isomorphic to the
dihedral group of order $8$ or the semi-direct product of
$\ZZ/2\rtimes \SS^1$.

\begin{proof}[Proofs of Corollaries \ref{cor:mu1}]
We use the fact that 
$X_{\Pi_0}$ admits a K\"ahler-Einstein 
metric \cite{Don08}.
Notice that although the group $G$ is disconnected, it acts trivially
on the cohomology of $\cM_t$ (cf. \S\ref{nonconnected}). So
it suffices to show that the relative Futaki invariant of $X_{\Pi_0}$ at the
K\"ahler-Einstein metric is nondegenerate in the sense \S\ref{nonconnected}. 
This follows trivially from
the fact that the space of $G$-invariant holomorphic vector fields on
$X_{\Pi_0}$ is reduced to $0$ since the action of $G$ does not fix any
point in  $\PP^1$. 
Thus $z'_{0,G}=0$ and the Futaki invariant is nondegenerate
necessarily nondegenerate.
We may apply Theorem \ref{theotech}, with the choice of space
$V=0$. Thus every class $\Theta_t$ is represented by a $G$-invariant extremal
K\"ahler metric $g_t$ for $t$ sufficiently small. Thus the holomorphic vector field $\grad s_{g_t}$ must
be $G$ invariant. As the action of $G$ does not fix any point in
$\PP^1$, it follows that there are no nontrivial $G$-invariant
holomorphic vector fields on $\cM_t$. Thus $s_{g_t}$ is constant
 and this finishes the proof of Corollary~\ref{cor:mu1}.
\end{proof}

\end{document}